\documentclass[oneside,10pt]{amsart}
\usepackage{amsmath}
\usepackage{amssymb}
\usepackage{amsthm}
\usepackage{float}
\usepackage{ulem}
\usepackage{xcolor}
\usepackage[utf8x]{inputenc}
\usepackage{comment}
\usepackage{framed}

\newcommand{\eps}{\varepsilon}

\DeclareMathAlphabet{\mathbbb}{U}{bbold}{m}{n}

\newtheorem{thm}{Theorem}[section]
\newtheorem{cor}[thm]{Corollary}
\newtheorem{lem}[thm]{Lemma}

\newtheorem{prop}[thm]{Proposition}
\theoremstyle{definition}
\newtheorem{defn}[thm]{Definition}
\theoremstyle{remark}
\newtheorem{rem}[thm]{Remark}
\newtheorem{ex}[thm]{Example}
\newtheorem{conj}[thm]{Open Question}
\numberwithin{equation}{section}

\usepackage{graphicx}

\usepackage{enumerate}

\newcommand{\R}{{\mathbb R}}

\newcommand{\matrice}{\begin{pmatrix}}
\newcommand{\ok}{\end{pmatrix}}

\newcommand{\Om}{\Omega}

\newcommand{\M}{\color{blue}}

\begin{document}
\title[]{On the critical points of semi-stable solutions on convex domains of  Riemannian surfaces}
\thanks{The first author acknowledges support of INdAM-GNAMPA. The second author acknowledges support of the INDAM-GNSAGA project ``Analisi Geometrica: Equazioni alle Derivate Parziali e Teoria delle Sottovarietà''.}
\author[Grossi]{Massimo Grossi }
\address{Dipartimento di Scienze di Base e Applicate per l'Ingegneria, Universit\`a di Roma ``La Sapienza'', Via Scarpa 12 - 00161 Roma, Italy, e-mail: {\sf massimo.grossi@uniroma1.it}.}
\author[Provenzano]{Luigi Provenzano}
\address{Dipartimento di Scienze di Base e Applicate per l'Ingegneria, Universit\`a di Roma ``La Sapienza'', Via Scarpa 12 - 00161 Roma, Italy, e-mail: {\sf luigi.provenzano@uniroma1.it}.}

\begin{abstract}
In this paper we consider semilinear equations $-\Delta u=f(u)$ with Dirichlet boundary conditions on certain convex domains of the two dimensional model spaces of constant curvature. We prove that a positive, semi-stable solution $u$ has exactly one non-degenerate critical point (a maximum). The proof consists in relating the critical points of the solution with the critical points of a suitable auxiliary function, jointly with a topological degree argument.
\end{abstract}

\keywords{Semilinear elliptic equations, critical points, convex domain, model spaces, Poincaré-Hopf Theorem}
\subjclass{58J32, 58J61, 58J20, 58J05}


\maketitle
\section{Introduction and statement of the main result}

Let $M$ be a two-dimensional model space, which means $M=\mathbb S^2,\mathbb R^2$ or $\mathbb H^2$ with the corresponding standard metrics of constant curvature $1,0,-1$. Let $\Omega\subset M$ be a bounded and smooth domain and let $u$ be a solution of the following Dirichlet problem
\begin{equation}\label{DL}
\begin{cases}
-\Delta u=f(u) &\hbox{in }\Om\,,\\
u=0 &\hbox{on }\partial\Om.
\end{cases}
\end{equation}
Here $f:[0,+\infty)\to\R$ is a $C^1$ nonlinearity satisfying $f(0)\ge0$. We will consider {\it positive} solutions of \eqref{DL} which are {\it semi-stable}. We say that $u$ is semi-stable if the first eigenvalue of the {\it stability operator}
\begin{equation}\label{s-s}
\mathcal{L}=-\Delta -f'(u)
\end{equation}
is non-negative. 

Two classical cases which fit with in this class are
\begin{itemize}
\item The {\bf torsion problem}:
\begin{equation}\label{f1}
f(s)=1
\end{equation}
\item The {\bf eigenvalue problem}, in particular, the {\bf first eigenfunction}:
\begin{equation}\label{f2}
f(s)=\lambda s
\end{equation}
where $\lambda$ is the first Dirichlet eigenvalue on $\Omega$.
\end{itemize}
There is a huge literature about the shape of solutions to \eqref{DL}. Indeed, the description of the geometry of the shape of the solutions, like the convexity of the super-level sets and the number of the critical points, is a problem that has engaged many mathematicians in the past decades.  On the other hand many questions are unsolved and a satisfactory description is available only in the $ flat$ case $\Omega\subset\R^2$ for $f$ as in \eqref{f1}, \eqref{f2}.
\subsection{The flat case $\Omega\subset\mathbb R^2$}
It is known that the shape of $\Omega$ influences the number of critical points of $u$. In the case of the torsion problem there are some conclusive answers. In \cite{ml}, Makar-Limanov proved that if $\Om$ is convex then the level sets of $u$ have positive curvature and the solution has only one critical point. The convexity assumption is difficult to relax: indeed in \cite{gg4} the authors show examples of domains ``close'' (in a suitable sense) to a convex one with a large numbers of critical points. The result of Makar-Limanov has been extended to any dimension by Korevaar and Lewis \cite{KoLe}.

Concerning the first eigenfunction of the Laplacian, among the first works on the subject we mention \cite{App} and \cite{bl}, where it is proved that if $\Omega$ is a strictly convex domain in $\mathbb R^2$, then the first eigenfunction is $\log$-concave. Some additional work is needed (see e.g., \cite{cf2}) to derive also the uniqueness of the critical point.  We point out that in \cite{bl}, log-concavity  of the first eigenfunction is proved in any dimension.

Next we mention two seminal papers for the case of a general nonlinearity: the first one is the celebrated paper by Gidas, Ni and Nirenberg \cite{gnn}, where it is proved, among other results, that there is uniqueness and non-degeneracy of the critical point
under the assumption that $\Om$ is symmetric with respect to a point and just convex in any direction. The second relevant result, which motivated our study, 
is \cite[Theorem 1]{cc} where the authors consider semi-stable solutions to semilinear elliptic equations and prove that for convex planar domains with boundary of positive curvature such solutions have exactly one non-degenerate critical point. The hypothesis of positive boundary curvature has been relaxed later in \cite[Theorem 2]{dgm}. We also mention \cite{dg} where the authors consider the second Dirichlet eigenfunction on certain planar convex sets.  
\subsection{The case of Riemannian surfaces}
In the Riemannian setting, much attention has been devoted to the case of the eigenfunctions. There is a quite vast literature on the properties of the nodal sets of eigenfunctions, mainly concerning their size (see e.g., \cite{df} and the review \cite{nty}). Much less is known about the number of critical points, and even more so  for a general nonlinearity $f$. An approach that allowed to get important results is to look for a metric $g$ on a manifold $M$ such that the corresponding $k$-th eigenfunction has a prescribed number of critical points.

For example, in \cite{eps} it is proved that, given a $n$-dimensional compact manifold $M$ with $n\geq 3$, there exits a metric $g$ such that for all positive integers $N$ and $l$, the $k$-th eigenfunction of the Laplacian has at least $N$ non-degenerate critical points, for  $k=1,…,l$. Similar results have been obtained in \cite{bls,epss,ms}.

A different point of view is to consider some {\it fixed} ambient two-dimensional manifold (e.g., $\mathbb S^2$ or $\mathbb H^2$ with their standard metrics) and study the convexity of the level sets of the first Dirichlet eigenfunction on  geodesically convex domains. In \cite{swyy}  the authors discuss some $\log$-concavity estimates for the first eigenfunction on convex Euclidean domains, providing a short proof of the result of Brascamp and Lieb \cite{bl}. This proof has been adapted in \cite{lw} to prove the the $\log$-concavity of the first Dirichlet eigenfunction in geodesically convex domains of $\mathbb S^n$. Note that this last result, jointly with \cite{w}, allows to prove the {\it uniqueness and non-degeneracy} of the critical point. For related results on positively curved surfaces, we also mention \cite{khan}.

Another problem which has been investigated  concerns the convexity of the level sets of eigenfunctions in the hyperbolic space $\mathbb H^2$. Here the situation is considerably different, Indeed, unlike the results in $\mathbb S^2$ or $\R^2$, there exist geodesically convex domains for which the first Dirichlet eigenfunction has non-convex level sets, see \cite{s}. Using similar ideas, in \cite{bcnsww} the author constructs an example of a convex domain such that the corresponding first Dirichlet eigenfunction has {\it two distinct} maxima. Hence the convexity of the domain is not enough to guarantee the uniqueness of the critical point of the eigenfunction in the negatively curved case.
\subsection{Statement of the main result}
The aim of this paper is to prove  the uniqueness of the critical point of solutions to \eqref{DL}  for certain natural classes of convex domains in $\mathbb S^2$ and $\mathbb H^2$. Our approach is suitable also to cover some other ambient spaces (see Remark \ref{other}), however, for the sake of presentation, we confine ourselves to the mentioned cases.  Moreover, our method provides an alternative simple proof of the uniqueness of the critical points for convex planar domains with positive boundary curvature, i.e., the result of \cite{cc} (see also \cite{dgm} for another alternative proof). Therefore we state it for the three model spaces.

Our main result is stated as follows:
\begin{thm}\label{main}Let $M$ be $\mathbb S^2,\mathbb R^2$ or $\mathbb H^2$ with the standard metric of constant curvature $1,0,-1$, and let $\Omega\subset M$ be a bounded and smooth domain. Assume that
\begin{enumerate}[i)]
\item $\Omega$ is convex with boundary of positive curvature if $M=\mathbb R^2$;
\item $\Omega$ is convex with boundary of positive curvature and diameter smaller than $\frac{\pi}{2}$ if $M=\mathbb S^2$;
\item $\Omega$ is horoconvex if $M=\mathbb H^2$ (the notion of horoconvexity is recalled in Definition \ref{horoconvex}). 
\end{enumerate}
Then any positive, semi-stable solution $u$ to \eqref{DL} has a unique non-degenerate critical point, which is a maximum.
\end{thm}
\begin{rem}
From Theorem \ref{main} we immediately get the following two corollaries.
\begin{cor}\label{torsion}Let $M$ be $\mathbb S^2,\mathbb R^2$ or $\mathbb H^2$ with the standard metric of constant curvature $1,0,-1$, and let $\Omega\subset M$ be a bounded and smooth domain. Let $u$ be the solution to
$$
\begin{cases}
-\Delta u=1 & {\rm in\ \Omega},\\
u=0 & {\rm on\ }\partial\Omega.
\end{cases}
$$ 
Assume that $i)$, $ii)$ of $iii)$ of Theorem \ref{main} hold, respectively, when $M=\mathbb R^2,\mathbb S^2$ or $\mathbb H^2$. Then $u$  has a unique non-degenerate critical point, which is a maximum.
\end{cor}
\begin{cor}\label{eigenfunction}Let $M$ be $\mathbb S^2,\mathbb R^2$ or $\mathbb H^2$ with the standard metric of constant curvature $1,0,-1$, and let $\Omega\subset M$ be a bounded and smooth domain. Let $u$ be the first eigenfunction of
$$
\begin{cases}
-\Delta u=\lambda u & {\rm in\ \Omega},\\
u=0 & {\rm on\ }\partial\Omega.
\end{cases}
$$ 
Assume that $i)$, $ii)$ of $iii)$ of Theorem \ref{main} hold, respectively, when $M=\mathbb R^2,\mathbb S^2$ or $\mathbb H^2$. Then $u$  has a unique non-degenerate critical point, which is a maximum if $u$ is chosen positive.
\end{cor}

We make some comments on the consequences of of Theorem \ref{main} and its corollaries.
\begin{itemize}
\item To our knowledge this is the first result in the literature on the uniqueness of the critical point for the {\it torsion problem} on manifolds. 
\item We remark that the additional hypotheses in $ii)$ and $iii)$ are crucial for our method to work, and just requiring $\kappa>0$ is not sufficient. In particular, they are crucial to prove Proposition \ref{geom}. Removing these hypotheses, we can easily find convex domains for which points $ii)$ and $iii)$ of Proposition \ref{geom} fail.

\item The condition on the diameter in $ii)$ is probably technical. In fact, in the case of the first eigenfunction we know (see \cite{lw}) that the maximum is unique for any convex domain, without any diameter restriction.
\item A natural question is whether it is possible to prove Theorem \ref{main} on the mere assumption of the convexity of $\Om$. This is not possible, as pointed out here above, in the case of the first Dirichlet eigenfunction of geodesically convex sets of $\mathbb H^2$ (but not horoconvex). Hence, in the hyperbolic case, the additional assumption of horoconvexity in $iii)$ does not seem to be just due to technical reasons.  
\item  As mentioned before, the results for the first Dirichlet eigenfunction on $\mathbb R^2$ and $\mathbb S^2$ are known. Corollary \ref{eigenfunction} implies that the first Dirichlet eigenfunction on horoconvex domains of $\mathbb H^2$ has a unique non-degenerate critical point. To the best of our knowledge, this is the first result of this kind for domains in negatively curved manifolds.
\end{itemize}
\end{rem}

\subsection{Strategy of the proof}
Now we give some ideas  about the proof of Theorem \ref{main}. Denote by $\mathcal{C}$ the set of critical points of $u$, namely
$$
\mathcal{C}=\{x\in M\hbox{ such that }\nabla u(x)=0\}.
$$
One of the main difficulties in describing the critical points of solutions to \eqref{DL} is that (a priori) the set $\mathcal{C}$ can have a complicated shape. It is not even  guaranteed that the $\mathcal{C}$ is finite, nor that its points are isolated. This is a serious problem if we want to apply classical tools as Morse theory or degree arguments. As we will show below, one of the main steps of our proof will be to prove that the set $\mathcal{C}$ consists of isolated points.

One of the most important tools that we use is the celebrated Poincar\'e-Hopf Theorem which links the index of the zeros of any vector field $V$ on a domain $\Om$ with the Euler characteristic of $\Om$  (see Section \ref{pre} and Theorem \ref{PH} for basic definitions and the statement of the results). In our setting, where $\Om$ is a contractible subset of a two-dimensional Riemannian manifold, it could be summarized by the formula
\begin{equation}\label{for}
\sum_i{\rm Ind}_{p_i}V= 1
\end{equation}
where the sum runs on the $isolated$ zeros of $V$. Usually \eqref{for} is applied to $V=\nabla u$, providing a balance on the critical points of $u$. Of course it says nothing about the exact number of the critical points of $u$. Actually  formula \eqref{for} will be applied to the following vector field
$$V=\nabla P=\nabla\left(\frac12|\nabla u|^2+F(u)\right)$$
where $F(s)=\int_0^sf(t)dt$ is the primitive of $f$. Note that this vector field was used in other context, see the $P$-functions in \cite{sp} or in \cite{yau2}, see also \cite{weinb}. At this stage two questions arise naturally,
\begin{enumerate}[1)]
\item \textit{Are the critical points of $P$ isolated}? 
\item  \textit{In which way the information on the number of critical points of $P$ allows to prove the uniqueness of the critical point of $u$}?
\end{enumerate}
The answer to the question $1)$ is the more delicate. It will be given in several steps
\begin{enumerate}[i)]
\item First we observe that critical points for $u$ are critical points for $P$.
\item We want to prove that the reciprocal implication is also true. Assume by contradiction that $p\in\Om$ is a not a critical point of $u$ and $\nabla P(p)=0$. Next we introduce an auxiliary function $Z:\Omega\to\mathbb R$ which is not identically zero and vanishes with its gradient at a point $p$ (this will be a consequence of the assumption $\nabla P(p)=0$). So $p$ is a \text{singular point} for $Z$ and classical results (see \cite{cf3} or \cite{hs}) imply that locally the zero-set of $u$ is given by a finite number of curves intersecting transversally. Then, we show that the convexity of  $\Om$ implies that $Z$ has exactly $two$ zeros on $\partial\Om$. We reach a contradiction by a topological argument: the function $Z$ defined in this way turns out to be a  Dirichlet eigenfunction of the stability operator $\mathcal L$ in a proper subdomain of $\Omega$ with eigenvalue $0$, and this implies that the first eigenvalue of $\mathcal L$ on $\Omega$ is strictly negative, contradicting the semi-stability assumption on $u$.
\item By the previous step we get that the number of critical points of $u$ coincides with the number of critical points of $P$.
Next we show that they are non-degenerate, and we do this in the same spirit of point $ii)$, using a suitable auxiliary function $W$.
\end{enumerate}

Once we know that the critical points of $u$  coincide with those of $P$ and they are not degenerate, we can apply the Poincar\'e-Hopf Theorem. Note that the crucial assumption $\langle \nabla P,\nu\rangle<0$ is verified by the fact that $\partial\Omega$ has positive curvature. A straightforward computation shows that the index of {\it any} critical point of $P$ is $one$  (this is a consequence of the non-degeneracy of the critical points of $u$) and so  by  \eqref{for}  we have
 \begin{equation}\label{for2}
\sharp\{\hbox{critical points of }P\}
=\sum_i{\rm Ind}_{p_i}\nabla(P)= 1
\end{equation}
 which gives the uniqueness (and by the previous discussion also the non-degeneracy) of the critical point of $P$, and then the same holds for $u$.

\begin{rem}\label{other}
We remark that the ideas used in this paper yield the same results in other situations, for example, it is straightforward to prove Theorem \ref{main} when $M$ is any coaxial cylinder in $\mathbb R^3$ or any flat torus, and $\Omega$ is a contractible convex set of $M$ with positive boundary curvature.
\end{rem}
The present paper is organized as follows. In Section \ref{pre} we collect a few preliminary results needed for the proof of Theorem \ref{main}, which is presented in Section \ref{proof}. In Appendix \ref{revolution} we restrict to the case of the first Dirichlet eigenfunction and we take another point of view, namely we consider the problem of describing the critical points on more general manifolds of revolution (with or without boundary) of any dimension.

{\bf Acknowledgments.} The authors would like to thank Gabriel Khan for pointing out an important list of references.

\section{Preliminaries and well known facts}\label{pre}
In this section we collect a few preliminary results and examples which will be useful in the proof of Theorem \ref{main}. Throughout the paper, for a Riemannian manifold $(M,g)$, we denote by $\langle\cdot,\cdot\rangle$ the inner product on the tangent spaces of $M$ associated with the metric $g$.

\subsection{Killing vector fields}
We start by recalling the definition of Killing vector field.

\begin{defn}\label{killing}
Let $(M,g)$ be a complete $n$-dimensional Riemannian manifold. A smooth vector field $K$ on $M$ is said to be {\it Killing} if, for every vector fields $X,Y$
\begin{equation}\label{Keq0}
\mathcal L_Kg(X,Y)=0,
\end{equation}
that is, the Lie derivative of $g$ with respect to $K$ vanishes.
\end{defn}
The {\it Killing equation} \eqref{Keq0} is equivalent to
\begin{equation}\label{Keq}
\langle \nabla_X K,Y\rangle+\langle \nabla_Y K,X\rangle=0.
\end{equation}
A further equivalent definition is the following: $K$ is Killing if the flow of $K$ is a local $1$-parameter group of isometries.

We recall a few consequences of Definition \ref{killing}. Let $K$ be a Killing vector field on $M$. Then
\begin{enumerate}
\item ${\rm div}K=0$;
\item if $K$ is pointwise tangential to an embedded submanifold $N$ of $M$, then $K|_{N}$ is a Killing vector field on $N$;
\item $\Delta K(u)=K(\Delta u)$ for any smooth function $u$, i.e., $K$ commutes with the Laplacian; also the reciprocal is true: if $K$ commutes with the Laplacian, then it is Killing.

\end{enumerate}
For more information on Killing vector fields  we refer to \cite[\S 8]{petersen}.

In the proof of Theorem \ref{main} the existence of global Killing fields will play a special role. Actually, we will ask more than the existence of some global Killing field. Namely, we will require that for any $p\in M$ and $v\in T_pM$ there exists a Killing field $K$ such that the geodesic $\gamma$ with $\gamma(0)=p$, $\gamma'(0)=v$ is an integral curve of $K$. This is true for the model spaces $\mathbb S^2,\mathbb R^2,\mathbb H^2$, as the following examples show.

\begin{ex}\label{euex}
Let $M=\mathbb R^2$. Three linearly independent Killing vector fields are given, in Cartesian coordinates $(x,y)$ by:
\begin{itemize}
\item $K_1=\partial_x$; the integral curves are lines parallel to the $x$-axis and in particular they are all geodesics;
\item $K_2=\partial_y$; the integral curves are lines parallel to the $y$-axis  and in particular they are all geodesics;
\item $-y\partial_x+x\partial_y$; the integral curves are circles about the origin and none of them is a geodesic.
\end{itemize}
Let $p\in\mathbb R^2$ and $v\in T_pM$. We can assume without loss of generality that $p=(0,0)$ and $v=(1,0)$. The geodesic $\gamma$ such that $\gamma(0)=p$ and $\gamma'(0)=v$ is just the $x$-axis of equation $y=0$. Then the Killing field $K$ having $y=0$ as integral curve is $K_1$.

An equivalent way of saying this is that, given a fixed system of Cartesian coordinates centered at $p$, the Killing vector field having a geodesic integral curve through $p$ is a linear combination of $K_1$ and $K_2$. The Killing field $K_3$ does not play any role since none of its integral curves is a geodesic.

\end{ex}

\begin{ex}\label{exsph}
Let $M=\mathbb S^2\subset\mathbb R^3$. Let $X=z\partial_y-y\partial_z$, $Y=z\partial_x-x\partial_z$ and $Z=-y\partial_x+x\partial_y$ be the three Killing vector fields in $\mathbb R^3$ (with Cartesian coordinates $(x,y,z)$) which are the generators of the rotations about the coordinate axes. They are tangential to $\mathbb S^2$ and therefore their restrictions to $\mathbb S^2$ are Killing vector fields on $\mathbb S^2$. Let $(\theta,\phi)\in[0,\pi]\times[0,2\pi]$ be the standard system of spherical coordinates where the north pole $(0,0,1)$ corresponds to $\theta=0$, while $(1,0,0)$ corresponds to $(\theta,\phi)=(\pi/2,0)$ and $(0,1,0)$ corresponds to $(\theta,\phi)=(\pi/2,\pi/2)$. Therefore three linear independent Killing fields are:
\begin{itemize}
\item $K_1=X|_{\mathbb S^2}=\sin(\phi)\partial_{\theta}+\cot(\theta)\cos(\phi)\partial_{\phi}$; the integral curves are spherical circles centered at $(\pm 1,0,0)$ and exactly one of such integral curves is a geodesic, namely the great circle $\mathbb S^2\cap\{x=0\}$;
\item $K_2=Y|_{\mathbb S^2}=\cos(\phi)\partial_{\theta}-\cot(\theta)\sin(\phi)\partial_{\phi}$; the integral curves are spherical circles centered at $(0,\pm 1,0)$ and exactly one of such integral curves is a geodesic, namely the great circle $\mathbb S^2\cap\{y=0\}$;
\item $K_3=Z|_{\mathbb S^2}=\partial_{\phi}$; the integral curves are spherical circles centered at $(0,0,\pm 1)$ and exactly one of such integral curves is a geodesic, namely the great circle $\mathbb S^2\cap\{z=0\}$, i.e., the equator.
\end{itemize}
Let $p\in\mathbb S^2$ and let $v\in T_p M$. We can assume without loss of generality that $p=(0,0,1)$ is the north pole and that $v=(0,1,0)$ (here we are thinking of $v$ as a vector in $\mathbb R^3$). Then the geodesic $\gamma$ through $p$ with $\gamma'(p)=(0,1,0)$ is the great circle $\mathbb S^2\cap\{x=0\}$, and consequently the Killing field having $\gamma$ as integral curve is $K_1$.  Also in this case, we note that if $p\in\mathbb S^2$ assuming that $p=(0,0,1)$, then any Killing field having as integral curve a geodesic through $p$ is a linear combination of  $K_1$ and $K_2$. Again, we note that, given a point $p$, the Killing field corresponding to the rotations around $p$, i.e., $K_3$, does not come into play.
\end{ex}
\begin{ex}\label{exhyp}
Let $M=\mathbb H^2$. Consider the Poincaré disk model for $\mathbb H^2$. Namely, we consider $D$ to be the open unit disk of $\mathbb R^2$ with Cartesian coordinates $(x,y)$ endowed with the metric $\frac{4}{(1-x^2-y^2)^2}(dx^2+dy^2)$.

Three linearly independent Killing vector fields are given by:
\begin{itemize}
\item $K_1=\frac{1-x^2+y^2}{2}\partial_x-xy\partial_y$; the integral curves are the intersection of $D$ with arcs of circles with centers on $x=0$ and passing through $(\pm 1,0)$; this include also the segment $(-1,1)\times\{0\}$ which is the unique geodesic integral line of $K_1$;
\item $K_2=-x y\partial_x+\frac{1-x^2+y^2}{2}\partial_y$; the integral curves are the intersection of $D$ with arcs of circles with centers on $y=0$ and passing through $(0,\pm 1)$; this include also the segment $\{0\}\times(-1,1)$ which is the unique geodesic integral line of $K_2$;
\item $K_3=-y\partial_x+x\partial_y$; the integral curves are circles centered at the origin and none of them is a geodesic.
\end{itemize}
Recalling that the geodesics in the Poincaré disk model are segments through the origin and arcs of circles in $D$ meeting $D$ orthogonally, we see that $K_1,K_2$ have only one integral curve which is a geodesic, while $K_3$ has none.

Let $p\in\mathbb H^2$. We can always consider the Poincaré disk model centered at $p$. Hence, without loss of generality we can assume that $v\in T_pM$ is given by $v=(0,1)$ in the disk model. Therefore the Killing vector field  having the (geodesic) segment $(-1,1)\times\{0\}$ as integral curve is $K_1$. Again, fixing a coordinate system centered at $p$, this amount to saying that a Killing field having as integral curve a geodesic through $p$ is a linear combination of $K_1$ and $K_2$, while $K_3$ does not play any role.
\end{ex}

\subsection{Poincaré-Hopf Theorem}
In this subsection we recall the Poincaré-Hopf Theorem, which relates the the zeros of a vector field with the Euler characteristic of the underlying manifold. In order to do so, we need some preliminary definitions.

\begin{defn}
Let $U\subset\mathbb R^n$ be an open set, $p\in U$, and let $\Phi:U\to\mathbb R^n$ be a continuous mapping such that $\Phi(p)=0$. Let $\eps>0$ be such that $B(p,\eps)\subset U$ does not contain other pre-images of the origin except $p$. The {\it local degree} of the map $\Phi$ at the point $p$ is defined as the degree of the mapping
$$
\frac{\Phi}{|\Phi|}:\partial B(p,\eps)\to\mathbb S^{n-1}
$$
where $\mathbb S^{n-1}$ is the unit sphere in $\mathbb R^n$.
\end{defn}

Let $M$ be a $n$-dimensional Riemannian manifold and let $V$ be a vector field on $M$. A point $p\in M$  such that $V(p)=0$ is called a {\it singular point} of $V$. In a system of local coordinates $x_1,...,x_n$ around $p$, we can write $V=\sum_{i=1}^nV_i(x)\partial_{x_i}$. We assume that the coordinates $(x_1,...,x_n)$ are defined in a neighborhood $U$ of the origin in $\mathbb R^n$, and that $p$ corresponds to the origin in $\mathbb R^n$.

\begin{defn}
Let $V$ be a vector field on $M$ and let $p$ be an isolated singular point of $V$. Let $V=(V_1,...,V_n)$ in a local coordinate system $(x_1,...,x_n)\in U\subset\mathbb R^n$ around $p$. The {\it index} ${\rm Ind}_{p}V$ of $V$ at $p$ is the local degree of the mapping $(V_1,...,V_n):U\to\mathbb R^n$
\end{defn}

We are ready to state the
\begin{thm}\label{PH}[Poincaré-Hopf Theorem]
Let $M$ be a $n$-dimensional Riemannian orientable manifold, with or without boundary, and let $V$ be a vector field on $M$ with isolated zeros $p_i$. If $\partial M\ne\emptyset$, assume that $\langle V,\nu\rangle <0$, i.e., $\langle V,\nu\rangle$ does not vanish (and has constant sign) on $\partial M$, where $\nu$ is the conormal vector to $\partial M$. Then
$$
\sum_i{\rm Ind}_{p_i}V= (-1)^n\chi(M)
$$
where $\chi(M)$ is the Euler characteristic of $M$.
\end{thm} 

For the proof of this result see e.g.,  \cite[page 2]{bss}.

Finally, we note that if a vector field $V$ on $M$ has a singular point at $p$ which is non-degenerate, i.e., ${\rm det}\left(\partial_{x_j}V_i(0)\right)\ne 0$, then ${\rm Ind}_pV={\rm sign}\,{\rm det}\left(\partial_{x_j}V_i(0)\right)$. Here $(x_1,...,x_n)$ is a local system of coordinates around $p$ (which corresponds to $0\in\mathbb R^n$) and $V=\sum_{i=1}^nV_i(x)\partial_{x_i}$.
\subsection{Nodal lines}
We end this section recalling a classical result (see \cite{hs}) on the behavior of solutions $v$ of elliptic equations on planar domains at a point $q$ where $v(q)=\nabla v(q)=0$ ($q$ is said to be a singular point). 
\begin{thm}\label{nod}
Suppose that $v$ is a non-constant solution to an elliptic equation of second order with smooth coefficients on a domain $\Omega\subset\mathbb R^2$. Then $v^{-1}\{0\}$ decomposes into the disjoint union $\left(v^{-1}\{0\}\cap\{|\nabla v|>0\}\right)\cup\left(v^{-1}\{0\}\cap|\nabla v|^{-1}\{0\}\right)$ of smooth one-dimensional manifolds having finite one-dimensional measure in each compact subset of $\Omega$, and the set of isolated singular points $v^{-1}(0)\cap\{|\nabla v|=0\}$.
\end{thm}

The previous theorem implies that, in a neighborhood of a singular point $q$,
 the zero-set of $v$ is given by (at least) two curves which intersect transversally. In this paper we are interested in the equation  $-\Delta v=a  v$ on two dimensional Riemannian manifolds, where $a$ is a smooth function. Using a local  coordinate chart around a singular point $q$ of $v$,  we deduce the following
\begin{cor}\label{cornod}
Let $v$ be a solution of $-\Delta v=a v$, on some domain $\Omega\subset M$, where $M$ is a two-dimensional Riemannian manifold and $a$ a smooth function. Let $q\in\Omega$ be such that $v(q)=\nabla v(q)=0$. Then, in a neighborhood of $q$ the set $v^{-1}(0)$ is given by (at least) two curves which intersect transversally.
\end{cor}

\section{Proof of Theorem \ref{main}}\label{proof}
\subsection{The auxiliary function $P$}
 Through all this section, $u$ is a positive, semi-stable solution of \eqref{DL}. We define
\begin{equation}\label{PF}
P:=\frac{1}{2}|\nabla u|^2+F(u).
\end{equation}
with $F(s)=\int_0^sf(t)dt$.

We have the following expression for the gradient of $P$:
\begin{equation}\label{DF-normal}
\nabla P=\nabla_{\nabla u}\nabla u-\Delta  u\nabla u.
\end{equation}

\subsection{The zeros of $\nabla P$ coincide with the zeros of $\nabla u$}\label{sub_Z}
From the definition of $P$, it turns out that if $p$ is a zero of $\nabla u$, then it is also a zero of $\nabla P$. The aim of this subsection is to show that the vice-versa holds true.

Let $p\in\Omega$ be such that $\nabla u(p)\ne 0$. Let $v\in T_pM$ be orthogonal to $\nabla u(p)$, and let $\gamma$ be the unique geodesic such that $\gamma(0)=p$, $\gamma'(0)=v$. Then $\gamma$ is the integral curve of some Killing vector field $K$ (see also Examples \ref{euex},\ref{exsph} and \ref{exhyp}). We  define
\begin{equation}\label{Z1}
Z:=K(u).
\end{equation}

\begin{lem}\label{commuting}
The function $Z$ does not vanish identically and satisfies
$$
-\Delta Z=f'(u) Z
$$
in $\Omega$. Moreover, $Z(p)=0$.
\end{lem}
\begin{proof}
The fact that $Z$ does not vanish identically is straightforward to check.  In fact, if $Z\equiv 0$ on $\Omega$, then $Z\equiv 0$ on $\partial\Omega$, which implies that $\partial\Omega$ is an integral curve of $K$ . On the other hand, the geodesic $\gamma$ through $p$ is an integral curve of $K$ which has non-empty intersection with $\partial\Omega$, which is not possible (alternatively, from Proposition \ref{geom} it follows that $Z$ has exactly two zeros on $\partial\Omega$, hence it cannot vanish identically in $\Omega$).

Since $K$ is a Killing vector field, it commutes with the Laplacian: $\Delta K(u)=K(\Delta u)$. This implies that $-\Delta Z=f'(u) Z$. The fact that $Z(p)=0$ follows just by construction, in fact $K$ is orthogonal to $\nabla u$ at $p$.
\end{proof}

\begin{ex}
In the case $M=\mathbb R^2$, we can assume without loss of generality that $p=(0,0)$ and $\nabla u(p)=(0,c)$ for some $c\ne 0$. Hence $v=(1,0)\in T_pM$ is orthogonal to $\nabla u(p)$. We are in the case of Example \ref{euex}, and Lemma \ref{commuting} simply says that $Z=u_x$ solves $-\Delta u_x=f'(u)u_x$, which is a trivial identity. Moreover, since $\nabla u(p)=(0,c)$, this implies that $Z(p)=u_x(p)=0$.

In the case of $M=\mathbb S^2\subset\mathbb R^3$, we can assume without loss of generality that $p=(0,0,1)$ and $\nabla u(p)=(1,0,0)$. Hence $v=(0,1,0)\in T_pM$ is orthogonal to $\nabla u(p)$. We are in the case of Example \ref{exsph}, which means that, in polar coordinates $(\theta,\phi)$ centered at the pole $p$, $Z=\sin(\phi)\partial_{\theta}u+\cot(\theta)\cos(\phi)\partial_{\phi}u$.  Lemma \ref{commuting} says that $-\Delta(\sin(\phi)\partial_{\theta}u+\cot(\theta)\cos(\phi)\partial_{\phi}u)=f'(u) (\sin(\phi)\partial_{\theta}u+\cot(\theta)\cos(\phi)\partial_{\phi}u)$, which can be easily verified since $\Delta=\partial^2_{\theta\theta}+\cot(\theta)\partial_{\theta}+\sin^{-2}(\theta)\partial^2_{\phi\phi}$.

Analogous explicit computations can be performed in the case of $\mathbb H^2$, using the explicit fields provided in Example \ref{exhyp}.
\end{ex}

Next we compute $\nabla Z(p)$.

\begin{lem}\label{DZ}
We have
\begin{equation}
\nabla Z(p)=\nabla_K\nabla u|_{p}.
\end{equation}
\end{lem}
\begin{proof}
We compute
$$
\nabla Z=\nabla K(u)=\nabla_K\nabla u+\nabla_{\nabla u}K.
$$
In order to prove the Lemma, we need to show that the second summand vanishes at $p$. Now, by the Killing equation \eqref{Keq}
$$ 
\langle \nabla_{\nabla u}K,K\rangle=-\langle \nabla_K K,\nabla u\rangle
$$
and $\nabla_K K|_{p}=0$, since this is the covariant derivative of the  tangent to a geodesic along the geodesic. On the other hand,  by the Killing equation we have
$$ 
\langle \nabla_{\nabla u}K,\nabla u\rangle=-\langle \nabla_{\nabla u} K,\nabla u\rangle
$$
and hence $\langle \nabla_{\nabla u}K,\nabla u\rangle=0$. Since at $p$ $(K,\nabla u)$ forms a orthogonal frame, we conclude that $\nabla_{\nabla u}K|_p=0$.
\end{proof}

Finally, we compare $\nabla Z(p)$ and $\nabla P(p)$, assuming that $\nabla u(p)\ne 0$.

\begin{lem}\label{nablaZ}
Let $p\in\Omega$ be such that $\nabla u(p)\ne 0$ and let $Z$ be defined by \eqref{Z1}. Then
$$
|\nabla P(p)|=|\nabla u(p)||\nabla Z(p)|\,,\ \ \ \ \langle \nabla P(p),\nabla Z(p)\rangle=0.
$$
\end{lem}
\begin{proof}
By construction, $\left(K,\frac{\nabla u}{|\nabla u|}\right)$ forms a orthonormal frame at $T_pM$ (one can check that $|K|=1$ along $\gamma$, and in particular, at $p$). Through the rest of the proof we will suppress the explicit dependence on $p$, since we will just look at $T_pM$. We use Lemma \ref{nablaZ} and deduce that, at $p$
\begin{equation}\label{DZ1}
\nabla Z=\langle \nabla_K \nabla u,K\rangle K+\langle \nabla_K\nabla u,\frac{\nabla u}{|\nabla u|}\rangle\frac{\nabla u}{|\nabla u|}.
\end{equation}
On the other hand, by \eqref{DF-normal} we have that, at $p$
\begin{equation}\label{DP}
\nabla P=\langle \nabla_{\nabla u}\nabla u-\Delta u\nabla u,K\rangle K+\left(\langle \nabla_{\nabla u}\nabla u,\frac{\nabla u}{|\nabla u|}\rangle-\Delta u|\nabla u|\right)\frac{\nabla u}{|\nabla u|}.
\end{equation}
However, observing that $\langle \nabla u,K\rangle=0$ at $p$, and that $\langle \nabla_{\nabla u}\nabla u,\frac{\nabla u}{|\nabla u|}\rangle-\Delta u|\nabla u|=-|\nabla u|\langle\nabla_K\nabla u,K\rangle$ (we use the fact that the Laplacian is the trace of the Hessian), we can rewrite \eqref{DP} as
\begin{multline}\label{DP1}
\nabla P=\langle \nabla_{\nabla u}\nabla u,K\rangle K-|\nabla u|\langle\nabla_K\nabla u,K\rangle\frac{\nabla u}{|\nabla u|}\\
=|\nabla u|\left(\langle \nabla_K\nabla u,\frac{\nabla u}{|\nabla u|}\rangle K-\langle\nabla_K\nabla u,K\rangle\frac{\nabla u}{|\nabla u|}\right).
\end{multline}
The conclusion follows from \eqref{DZ1} and \eqref{DP1}.

\end{proof}

We have defined, for any $p\in\Omega$ with $\nabla u(p)\ne 0$, a function $Z$ such that $-\Delta Z=f'(u)Z$ in $\Omega$ and $Z(p)=0$. Note that the Killing vector field used to define $Z$ in \eqref{Z1} depends on $p$. We will use the function $Z$ to prove that under some geometric conditions on $\Omega$, $P$ and $u$ have the same critical points. This is contained in the next proposition. A key ingredient in its proof is that, under certain geometric conditions, $Z$ has exactly two zeros on $\partial\Omega$. The proof of this last fact is postponed to Subsection \ref{gcon}.

\begin{prop}\label{samezeros}
Let $\Omega\subset M$ be a bounded and smooth domain. Assume that
\begin{enumerate}[i)]
\item $\Omega$ is convex with boundary of positive curvature if $M=\mathbb R^2$;
\item $\Omega$ is convex with boundary of positive curvature and diameter smaller than $\frac{\pi}{2}$ if $M=\mathbb S^2$;
\item $\Omega$ is horoconvex if $M=\mathbb H^2$.
\end{enumerate}
Then $p$ is a zero of $\nabla P$ if and only if it is a zero of $\nabla u$.
\end{prop}
\begin{proof}
It is straightforward to check that if $p\in\Omega$ is such that $\nabla u(p)=0$, then also $\nabla P(p)=0$. On the other hand, assume that $p\in\Omega$ is a zero of $\nabla P$, but $\nabla u(p)\ne 0$. Therefore we can define a function $Z$ as in \eqref{Z1}. It follows by Lemma \ref{commuting} that $Z$ does not vanish identically, $Z(p)=0$ and $-\Delta Z=f'(u)Z$ and by Lemma \ref{nablaZ} that $\nabla Z(p)=0$, since $\nabla P(p)=0$. Hence by Corollary \ref{cornod} we deduce that in a neighborhood of $p$, $v^{-1}(0)$ is given by (at least) two curves which intersect transversally.  Moreover, hypotheses $i),ii)$ and $iii)$ and Proposition \ref{geom} imply that $Z$ has exactly two zeros on $\partial\Omega$. This implies  that the set $Z=0$ creates a loop, i.e., there exists an open set $\omega\subset\subset\Omega$ such that $Z=0$ on $\partial\Omega$ and $Z$ does not change sign in $\omega$, and moreover solves  $-\Delta Z=f'(u)Z$ in $\omega$. Hence, by domain monotonicity, the first eigenvalue of the operator $-\Delta-f'(u)$ in $\Om$ is negative and this is a contradiction with the semi-stability of the solution $u$.
\end{proof}

\subsection{The critical points of $u$ are non-degenerate}\label{sub_W}
An argument analogous to the one of the previous subsection allows to prove that the critical points of $u$ are non-degenerate.

Suppose that $u$ is such that $\nabla u(p)=0$, and let $v\in T_pM$ be such that, at $p$, $D^2u(v,w)=0$ for all $w\in T_pM$, that is, $p$ is a degenerate critical point.

As for the definition of $Z$, let $K$ be a Killing vector field such that the geodesic $\gamma$ with $\gamma(0)=p$, $\gamma'(0)=v$ is an integral curve of $K$. Then we define
\begin{equation}\label{W1}
W:=K(u).
\end{equation}
We have the analogous of Lemma \ref{commuting}.
\begin{lem}\label{commuting2}
The function $W$ does not vanish identically and satisfies
$$
-\Delta W=f'(u)W
$$
in $\Omega$. Moreover, $W(p)=0$.
\end{lem}
The fact that $W$ does not vanish identically is straightforward to check (see the proof of Lemma \ref{commuting}). The fact that $W(p)=0$ is trivial, in fact by hypothesis $p$ is a critical point of $u$.
\begin{lem}\label{nablaW}
Let $p\in\Omega$ be a degenerate critical point. Then
$$
\nabla W(p)=0.
$$
\end{lem}
\begin{proof}
 Let $W$ be defined as in \eqref{W1}. Hence, at $p$, $D^2u(K,X)=0$ for all $X\in T_pM$. From Lemma \ref{DZ} (with $Z$ replaced by $W$) we have than, for all $X\in T_pM$:
$$
\langle\nabla W,X\rangle=\langle\nabla_K\nabla u,X\rangle=D^2u(K,X)=0.
$$
\end{proof}

\begin{prop}\label{nodeg}
Let $\Omega\subset M$ be a bounded and smooth domain. Assume that
\begin{enumerate}[i)]
\item $\Omega$ is convex with boundary of positive curvature if $M=\mathbb R^2$;
\item $\Omega$ is convex with boundary of positive curvature and diameter smaller than $\frac{\pi}{2}$ if $M=\mathbb S^2$;
\item $\Omega$ is horoconvex if $M=\mathbb H^2$.
\end{enumerate}
Then the zeros of $\nabla u$ are non-degenerate critical points of $u$.
\end{prop}
\begin{proof}
Assume that $p\in\Omega$ is a degenerate critical point of $u$. Then we can define a function $W$ as in \eqref{W1}. It follows by Lemma \ref{commuting2} that $W$ does not vanish identically, $W(p)=0$ and $
-\Delta W=f'(u)W$, and by Lemma \ref{nablaW} that $\nabla W(p)=0$. Moreover, hypotheses $i),ii)$ and $iii)$ and Proposition \ref{geom} imply that $W$ has exactly two zeros on $\partial\Omega$. As in Proposition \ref{samezeros}, this implies that there exists an open set $\omega\subset\subset\Omega$ such that $W=0$ on $\partial\omega$ and does not change sign on $\omega$, and moreover solves  $-\Delta W=f'(u)W$ in $\omega$. We conclude as in Proposition \ref{samezeros}.
\end{proof}

\subsection{Application of Poincaré-Hopf Theorem and conclusion of the proof}\label{sub_conclusion}

In order to conclude, we want to apply the Poincaré-Hopf Theorem to the vector field $\nabla P$.

In order to do so we need to compute ${\rm Ind}_{p}\nabla P$ for all critical points $p$ of $P$ and  the sign of $\langle\nabla P,\nu\rangle$ at $\partial\Omega$. On the other hand, since we are considering convex domains $\Omega$ in $M$, we always have $\chi(\Omega)=1$. 

\begin{lem}
Let $\nu$ be the outer unit normal to $\partial\Omega$. Then
\begin{equation}\label{bdry_curv}
\langle \nabla P,\nu\rangle=-|\nabla u|^2\kappa
\end{equation}
where $\kappa$ is the geodesic curvature of $\partial\Omega$ (with respect to the the orientation given by $\nu$). 
\end{lem} 
\begin{proof}
Recall that at $\partial\Omega$, $\nu=-\frac{\nabla u}{|\nabla u|}$ (we assume $u>0$ in $\Omega$, but nothing essentially changes if we take $u<0$), and that we can write, for a vector field $X$, $\langle \nabla_{\nabla u}\nabla u,X\rangle=D^2u(\nabla u, X)$, where $D^2u$ is the Hessian of $u$.
Hence
\begin{multline*}
\langle \nabla P,\nu\rangle=-\langle \nabla_{\nabla u}\nabla u-\Delta u\nabla u,\frac{\nabla u}{|\nabla u|}\rangle=-|\nabla u|D^2u\left(\frac{\nabla u}{|\nabla u|},\frac{\nabla u}{|\nabla u|}\right)+\Delta u|\nabla u|\\
=|\nabla u|(-\partial^2_{\nu\nu}u+\Delta u)=|\nabla u|(-\partial^2_{\nu\nu}u+\Delta_{\partial\Omega} u+\kappa\partial_{\nu}u+\partial^2_{\nu\nu}u)=-|\nabla u|^2\kappa,
\end{multline*}
where $\kappa$ is the geodesic curvature of the boundary with respect to the the orientation given by $\nu$. Here we have used the well-known decomposition $\Delta u|_{\partial\Omega}=\Delta_{\partial\Omega}u+\kappa\partial_{\nu}u+\partial^2_{\nu\nu}u$, where $\Delta_{\partial\Omega}$ denotes the Laplacian on $\partial\Omega$, $\partial_{\nu}u=\langle\nabla u,\nu\rangle$ and $\partial^2_{\nu\nu}u=D^2u(\nu,\nu)$. Since $u=0$ on $\partial\Omega$, $\Delta_{\partial\Omega}u=0$.
\end{proof}

Now we compute the index of the critical points of $P$.

\begin{lem}\label{index_crit}
Let $p\in\Omega$ be such that $\nabla P(p)=0$. Then ${\rm Ind}_p \nabla P=1$.
\end{lem}
\begin{proof}
 We consider $D^2P:=\nabla \nabla P$ (is the iterated covariant derivative). Consider any system of local coordinates $(x_1,x_2)$ in a neighborhood of $p$. Since $\nabla P(p)=0$, at $p$ we can write $D^2P$ in coordinates as
$$
D^2P=(D^2u(p)-\Delta u(p)I)D^2u(p)
$$
where  $I$ is th $2\times 2$ identity matrix and
$$
D^2u(p)=(\partial^2_{x_ix_j}u)_{ij}
$$
It is now immediate to check that ${\rm det}D^2P=({\rm det}D^2u(p))^2>0$ since $p$ is a non  degenerate critical point for $u$ (see Proposition \ref{nodeg}). This follows from the fact that for any $2\times 2$ matrix $A$  we have ${\rm det}((A-{\rm Tr}(A) I)A)={\rm det}(A-{\rm Tr}(A) I){\rm det}(A)=({\rm det}(A))^2$.  This concludes the proof.

\end{proof}

We can now conclude the proof of Theorem \ref{main}. Let $p_i$ denote the critical points of $P$. From Lemma \ref{bdry_curv} we deduce that $\langle\nabla P,\nu\rangle<0$ on $\partial\Omega$ under hypotheses $i)$,$ii)$, $iii)$,  and also by the fact that $|\nabla u|>0$ on $\partial\Omega$. This last fact follows by Hopf's Lemma which applies to $u$ since we have assumed that the nonlinearity $f$ satisfies $f(0)\geq 0$. Then by Theorem \ref{PH} we have
$$
\sum_{p_i}{\rm Ind}_{p_i}\nabla P=\chi(\Omega)=1
$$
since $\Omega$ is convex.  From Lemma \ref{index_crit} we deduce that ${\rm Ind}_{p_i}\nabla P=1$ for all $p_i$. Therefore $\nabla P$ has only one critical point $p$. From Proposition \ref{samezeros} we deduce that $p$ is the unique critical point of $u$, and from Proposition \ref{nodeg} we deduce that it is non-degenerate. The proof of Theorem \ref{main} is concluded.

\subsection{Geometric conditions on the convex sets}\label{gcon}

In this subsection we prove that under conditions $i)$-$iii)$ of Theorem \ref{main} the functions $Z$ and $W$ defined respectively in \eqref{Z1} and \eqref{W1} have exactly two zeros on $\partial\Omega$.

We first recall a few concepts of convexity in the hyperbolic plane $\mathbb H^2$.

\begin{defn}
A horocycle is a continuous curve in $\mathbb H^2$ whose normal geodesics all converge asymptotically in the same direction. Horocycles have constant geodesic curvature $\kappa=1$.
\end{defn}

\begin{defn}\label{horoconvex}
We say that a domain $\Omega\subset\mathbb H^2$ is horoconvex if at every point $p\in\partial\Omega$ there exists a horocycle passing through $p$ such that $\Omega$ is contained in the region bounded by the horocycle.
\end{defn}

In the Poincaré disk model, which is the model of Hyperbolic geometry that we are considering in this article, horocycles are Euclidean circles entirely contained in $D$ and tangent to $\partial D$ .

\begin{prop}\label{geom}
Let $\Omega\subset M$ be a bounded and smooth domain, let $u:\Omega\to\mathbb R$ be such that $u=0$ on $\partial\Omega$, $\nabla u\ne 0$ on $\partial\Omega$. Let $p\in\Omega$, $v\in T_pM$, $\gamma$ a geodesic with $\gamma(0)=p$, $\gamma'(0)=v$, and let $K$ be a Killing vector field such that $\gamma$ is an integral curve of $K$. Assume moreover that
\begin{enumerate}[i)]
\item $\Omega$ is convex with boundary of positive curvature if $M=\mathbb R^2$;
\item $\Omega$ is is convex with boundary of positive curvature and diameter smaller than $\frac{\pi}{2}$ if $M=\mathbb S^2$;
\item $\Omega$ is horoconvex if $M=\mathbb H^2$.
\end{enumerate}
Then the function $F:=K(u)$ has exactly two zeros on $\partial\Omega$.
\end{prop}
\begin{proof}
We start with the simple case $i)$. We can assume without loss of generality that $p=(0,0)$ and that $v=(1,0)$, so that $K=\partial_x$ (here we are using Cartesian coordinates $(x,y)$ in $\mathbb R^2$). 

Since $\nabla u\ne 0$ on $\partial\Omega$, we have that $\nu=-\frac{\nabla u}{|\nabla u|}$   is the outer unit normal to $\partial\Omega$ (assuming $u>0$). Hence the condition $K(u)=0$ at $s\in\partial\Omega$ translates on the geometric condition that at $s$ the integral curve of $K$ is tangent to $\partial\Omega$. Now, the integral curves of $K$ are lines parallel to the $x$-axis. Therefore only two such lines are tangent to $\partial\Omega$, being $\Omega$ strictly convex by hypothesis $i)$.

We pass to the proof of $ii)$. The assumption on the diameter of $\Omega$ implies that $\Omega$ is strictly contained in a hemisphere centered at $p$, for any $p\in\Omega$. Let $p$ be as in the statement. It is convenient to think of $\mathbb S^2$ as embedded in $\mathbb R^3$:  $\mathbb S^2\subset\mathbb R^3$, $\mathbb S^2:=\{(x,y,z)\in\mathbb R^3:x^2+y^2+z^2=1\}$. Without loss of generality, we can assume that $p=(0,0,-1)$ is the south pole, and hence that $\Omega$ is contained in the lower hemisphere $\mathbb S^2_-:=\mathbb S^2\cap\{z<0\}$. The geodesic $\gamma$ is a great circle passing through $p$, and hence $K$ is a field whose integral lines are spherical circles centered at two antipodal points on the equator $\mathbb S^2\cap\{z=0\}$. In view of Example \ref{exsph}, we can assume without loss of generality that $K=\sin(\phi)\partial_{\theta}+\cot(\theta)\cos(\phi)\partial_{\phi}$.

Now, consider the central projection $\Pi$ of $\Omega$ and of the level curves of $K$ on the plane $\pi:\{z=-1\}$.  Recall that, for $s\in\mathbb S^2_-$, $\Pi(s)$ is the intersection of the line through the origin and $s$ with the plane $\pi$.  We identify in a natural way points $(x,y,-1)\in\pi$ with points $(x,y)\in\mathbb R^2$ and hence $\pi$ with $\mathbb R^2$. If $\Omega\subset\mathbb S^2_-$ is strictly convex, then also $\Pi(\Omega)$ is strictly convex in $\mathbb R^2$ and moreover it contains the origin $\Pi(p)$. Finally, the central projections of the integral curves of $K$ foliate $\mathbb R^2$ by hyperbolas of equation $y^2-\frac{1-a^2}{a^2}x^2+1=0$, $a\in[0,1)$. When $a=0$ we have the projection of the arc of great circle $\mathbb S^2_+\cap\{x=0\}$ which is just the $y$-axis of equation $x=0$ in $\mathbb R^2$.

Our problem is then re-formulated in the following terms: let $\Omega\subset\mathbb R^2$ be a bounded and strictly convex domain of $\mathbb R^2$ containing the origin. Then only two branches of hyperbolas $H_a:=\{(x,y):y^2-\frac{1-a^2}{a^2}x^2+1=0\}$, $a\in(0,1)$ are tangent to $\partial\Omega$, each one in exactly one point.

 Consider first $\Omega\cap \{x>0\}$. Since $0\in\Omega$, we have that there exists at least one hyperbola $H_{a_1}$ to which $\Omega$ is tangent. Suppose by contradiction that they are at least two: $H_{a_1},H_{a_2}$.

 Observe that if $a_1\ne a_2$ then $H_{a_1}$ and $H_{a_2}$ do not intersect. Moreover, by the strict convexity of $\Om$ and the interior of $H_{a_1}$, there is one straight line $r_1$ which separates $\Om$ and $H_{a_1}$. Hence, if $a_1<a_2$ the hyperbola $H_{a_2}$ lies on the right of $H_{a_1}$ and so cannot intersect $\Om$. Of course this gives that there is no hyperbola $H_{a_2}$ tangent to $\Om$ for $a_1<a_2$. In the same way we see that if $a_2<a_1$ then $H_{a_1}$ does not intersect $\Omega$. A contradiction.

It remains to consider the case $a_1=a_2=:a$. Assume that $\Omega$ is tangent to two distinct points on the same hyperbola $H_a$, say $q_1,q_2$. Then $\Omega$ contains the whole segment $[q_1,q_2]$, and since the origin belongs to $\Omega$, it contains also the arc of $H_a$ between $q_1$ and $q_2$, and hence $\Omega$ is not tangent to $H_a$ at $q_1,q_2$. A contradiction.

We conclude that for $x>0$ only one hyperbola $H_a$ is tangent to $\Omega$ at exactly one point. The same is true for $x<0$.

We conclude with $iii)$. Let $\Omega,p$ be as in the statement. Without loss of generality we can assume that $p=(0,0)$ in the Poincaré disk model, and that $K=-\frac{1-x^2+y^2}{2}\partial_x+xy\partial_y$. It is not difficult to see that the level curves of $K$ are given by $C_a:=\{x^2+(y-a)^2=1+a^2\}\cap D$, $a\in\mathbb R\setminus\{0\}$.

Recall that here $D$ is the unit disk in $\mathbb R^2$. The level curves $C_a$ are the portions of circles centered at $(0,a)$ of radius $\sqrt{1+a^2}>1$ contained in $D$. On the other hand, $\Omega$ is a horoconvex domain containing the origin. This means that it is contained in a horodisk (a disk bounded by a horocycle) tangent at any of its points. This implies that $\Omega$ is Euclidean convex with boundary curvature $\geq 1$. Now we proceed similarly to $ii)$ and prove that for $y>0$ we have only one $C_a$, $a<0$, tangent to $\Omega$. Clearly, we have at least one. Now, if we have two points of tangency $q_1,q_2$ on the same $C_a$, then $\Omega$ is contained, and tangent, to two horocycles which are also tangent to $C_a$ at $q_1,q_2$. This would imply that $q_1$ does not belong to the horocycle tangent at $q_2$, so it belongs to the complement of $\overline\Omega$ (and, in the same way, $q_2$ does not belong to the horocycle tangent at $q_1$), and this is not possible. If the two tangency points belong to two different $C_a$, say $C_{a_1},C_{a_2}$, assume that $|a_2|>|a_1|$. Then $\Omega$ is supported and tangent to a horocycle contained entirely in the disk centered at $(0,a_2)$ of radius $\sqrt{1+a_2^2}$. Hence $\Omega$ cannot be tangent to $C_{a_1}$. A contradiction. This concludes the proof.
\end{proof}

We conclude with a final remark.
\begin{rem}
It is natural to ask if the previous ideas could be applied to more general surfaces. Although some results can be applied, it does not seem possible to have immediate generalizations. Let us try to describe a possible strategy for closed surfaces in the case of the first non-trivial eigenfunction (i.e., $f(s)=\lambda s$, $\lambda$ the first non-trivial eigenvalue). 

Let $(M,g)$ be a closed Riemannian surface with the property that for any $p\in M$ and $v\in T_pM$ there exists a (global) Killing vector field such that the geodesic $\gamma$ with $\gamma(0)=p$, $\gamma'(0)=v$ is an integral curve of $K$ (a necessary condition is that there exist at least two Killing vector fields which are linearly independent at each point of $M$). Then the strategy of the proof of Theorem \ref{main} works in this case in a more straightforward way and applies to the second eigenfunction of the Laplacian on $M$: let $u$ be a solution of $-\Delta u=\lambda u$, where $\lambda>0$ is the second eigenvalue (the first is zero). Let us define a function $P$ as in \eqref{PF} and let $Z$ be defined as in \eqref{Z1}. Then $Z$ is a second eigenfunction of the Laplacian on $M$ if it is not identically $0$. If $p$ is a zero of $P$ but not of $\nabla u$, we prove as in Subsection \ref{sub_Z} that $Z$ and its gradient vanish at $p$. This implies that at $p$ we have two nodal lines of $Z$ meeting transversally, and therefore, that $Z$ has at least three nodal domains. A contradiction. Hence critical points of $u$ are critical points of $P$ and vice-versa. Note that here we did not have to check the vanishing of $Z$ at two points of the boundary, being $M$ boundaryless. In the same way, arguing as in Subsection \ref{sub_W}, we conclude that the critical points of $u$ are non-degenerate. In this situation, we can conclude as in Subsection \ref{sub_conclusion}:  for example, in the orientable genus $0$ case we have $\chi(M)=2$ and hence $P$, and therefore $u$, have exactly two non-degenerate critical points. Necessarily, the critical points of $u$  are a maximum and a minimum.

Unfortunately this approach fails at some points:
\begin{itemize}
\item There are few surfaces with two global Killing vector fields and in those cases  the second eigenfunctions are known explicitly (round sphere, flat torus, real projective plane, etc., see e.g., \cite{myers}).
\item Even if two global Killing vector fields are available, it is not clear if the function $Z$ which we define in \eqref{Z1}  is not identically zero (this applies also to the function $W$ in \eqref{W1}). This may depend on the eigenfunction $u$ and the point $p$ defining $Z$ (recall that $Z$ depends on $u$ and $p$, see Subsection \ref{sub_Z}). For example it happens for the round sphere if, in polar coordinates $(\theta,\phi)$, the eigenfunction is $u=\cot(\theta)$ and $p=(\pi/2,\phi_0)$. In this case  the Killing vector field defining $Z$ is $K=\frac{\partial}{\partial\phi}$ and therefore $Z=\partial_{\phi}u\equiv 0$. However we know  that $u$ has two critical points. It may also happen that $Z\equiv 0$ and $u$ has no isolated critical points but a one dimensional set of critical points. This happens, for example, in the case of the flat torus.
\end{itemize}
\end{rem}

\appendix

\section{On the eigenvalue problem for manifolds of revolution}\label{revolution}
In this Appendix we collect a few information on the critical points of solutions to \eqref{DL} when $f(s)=\lambda s$ and $\lambda$ is the first (non-trivial) eigenvalue of the Laplacian.  We limit ourselves to considering the case of manifolds of revolution, with or without boundary of any dimension $n$. In particular, we will end our analysis with a conjecture for closed manifolds.

We emphasize that rotationally invariant metrics are somehow special, and for such metrics it is easier to obtain results also in higher dimensions. To this regard, we mention \cite[Theorem 2]{cc}, where the authors prove that for strictly convex domains of revolution around an axis in $\mathbb R^n$, a positive semi-stable solution to a semi-linear elliptic equation admits a unique non-degenerate critical point which is a maximum.

We recall that a simply connected $n$-dimensional Riemannian manifold $(M,g)$ with a distinguished point $x_0$ is called a {\it revolution manifold with pole $x_0$} if $M\setminus\{x_0\}$ is isometric to $(0,D)\times\mathbb S^{n-1}$  and its metric, in polar coordinates $(r,t)\in(0,D)\times\mathbb S^{n-1}$ based at $x_0$, is written as $g=dr^2+\Theta^2(r)g_{\mathbb S^{n-1}}$, $\Theta(0)=0$, $\Theta>0$ in $(0,D)$. We can assume that $\Theta'(0)=1$. We call $D$ the diameter of $M$. The density of the Riemannian metric on $M$ in polar coordinates is given by $\sqrt{{\rm det}g}=\Theta^{n-1}(r)$.

  For space forms of constant curvature $K=1,0,-1$ we have
$$
\Theta(r)=\begin{cases}
\sin(r)\,, & K=1,\\
r\,, & K=0,\\
\sinh(r)\,, & K=-1.
\end{cases}
$$
We recall that when $n=2$ the quantity $-\frac{\Theta''}{\Theta}$ is the Gaussian curvature of $M$. In general, $-\frac{\Theta''}{\Theta}$ is the sectional curvature $\mathcal K(v_i,v_n)$, $i=1,...,n-1$, where $v_n=\partial_r$ and $v_i$ are the coordinate fields on $\mathbb S^{n-1}$.

\subsection{Manifolds of revolution with boundary}

We consider here the case of a Riemannian manifold of revolution with connected boundary $\partial M$, and $u$ the first eigenfunction of the Dirichlet Laplacian on $M$:
\begin{equation}\label{rev_D}
\begin{cases}
-\Delta u=\lambda u\,, & {\rm in\ }M,\\
u=0\,, & {\rm on\ }\partial M,
\end{cases}
\end{equation}
where $\lambda>0$ is the first eigenvalue. We have that the critical point is unique, non-degenerate, and it is a maximum if $u$ is chosen positive, without any assumptions on the rotationally invariant metric $g$.

\begin{prop}\label{rev_bdry}
Let $(M,g)$ be a simply connected manifold of revolution with  boundary. Then the first Dirichlet eigenfunction on $M$ has a unique critical point. If $u$ is positive, it is a maximum.
\end{prop}
\begin{proof}
We have that $D<+\infty$ and that $\partial M$ is homothetic to $\mathbb S^{n-1}$. 
 It is standard to see that the eigenfunctions of \eqref{DL} on $M$ are written in polar coordinates $(r,t)\in(0,D)\times\mathbb S^{n-1}$ as $u_{k,l}(r)H_l(t)$, for $l\in\mathbb N$, $k=1,2,...$, where $H_l(t)$ is some spherical harmonic of degree $l$ in $\mathbb S^{n-1}$. For each $l\in\mathbb N$, $u_{k,l}$ solves
\begin{equation}\label{SL1}
\begin{cases}
-u_{k,l}''(r)-(n-1)\frac{\Theta'(r)}{\Theta(r)}u_{k,l}'(r)+\frac{l(l+n-2)}{\Theta^2(r)}u_{k,l}(r)=\lambda_{k,l} u_{k,l}(r)\,, & r\in(0,D)\\
u_{k,0}'(0)=0{\rm\ if\ }l=0{\rm \ and\ }u_{k,l}(0)=0{\rm \ if\ }l\ne 0\\
u_{k,l}(D)=0.
\end{cases}
\end{equation}
For each fixed $l$, problem \eqref{SL1} admits an increasing sequence of positive eigenvalues $\{\lambda_{k,l}\}_k$ diverging to $+\infty$, and a corresponding orthonormal basis $\{u_{k,l}\}_k$ of $L^2((0,D),\Theta^{n-1}(r)dr)$ of eigenfunctions. The Dirichlet spectrum of $M$ is given by the union of the spectra of \eqref{SL1}, namely, by $\{\lambda_{k,l}\}_{k,l}$.

\medskip

In particular, denoting by $u,\lambda$ the first Dirichlet eigenfunction and eigenvalue of $M$, we have that
$$
u(r,t)=u_{1,0}(r)\,,\ \ \ \lambda=\lambda_{1,0}.
$$
In particular $u$ does not change sign on $(0,D)$ and solves
$$
-\frac{1}{\Theta^{n-1}}(\Theta^{n-1} u')'=\lambda u\ \ \ r\in(0,D)
$$
Assume that $u>0$. Then $(\Theta^{n-1} u')'<0$, which means that $\Theta^{n-1} u'$ is decreasing. Now, $\Theta^{n-1}(0)u'(0)=0$, which implies that $u'<0$, i.e., $u$ is strictly decreasing.

Therefore we conclude that $u$ has a unique critical point, which is a maximum. Moreover, $u$ is radially symmetric and strictly decreasing in the radial variable.
\end{proof}

\subsection{Closed manifolds of revolution}

In the case of a closed manifold of revolution diffeomorphic to $\mathbb S^n$, we consider an eigenfunction $u$ associated with the second eigenvalue $\lambda>0$ of the Laplacian on $M$, namely:
\begin{equation}\label{rev_C}
-\Delta u=\lambda u \ \ \ {\rm in\ }M.
\end{equation}
The first eigenvalue of the Laplacian on $M$ is $0$, with corresponding constant eigenfunctions. We are interested in the critical points of the second eigenfunction. The goal is to prove that it has exactly two non-degenerate critical points, which are a maximum and a minimum. We give  positive answers under certain conditions on the metric $g$, which, in dimension $n=2$ are equivalent to requiring that the Gaussian curvature of $M$ is positive. We denote by ${\rm mult}(\lambda)$ the multiplicity of the second eigenvalue. For manifolds of revolution diffeomorphic to the sphere we have ${\rm mult}(\lambda)\leq n+1$.

\begin{prop}\label{rev_closed}
Let $(M,g)$ be a closed manifold of revolution with metric  $g=dr^2+\Theta^2(r)g_{\mathbb S^{n-1}}$. Assume that $-\frac{\Theta''}{\Theta}>0$ in $M$. If ${\rm mult}(\lambda)\ne n+1$ then any eigenfunction of the Laplacian on $M$ associated with the second eigenvalue $\lambda$ has two non-degenerate critical points, a maximum and a minimum. If ${\rm mult}(\lambda)= n+1$, there exists a basis of a corresponding eigenspace of eigenfunctions with two non-degenerate critical points, a maximum and a minimum. 
\end{prop}
\begin{proof}
We consider polar coordinates $(r,t)\in(0,D)\times\mathbb S^{n-1}$, where $D$ (the diameter) is the distance between $x_0$ (the pole) and its opposite $-x_0$. Moreover $\Theta(D)=0$. The  spectrum of $M$ is given by the union $\{\lambda_{k,l}\}_{k,l}$ of the spectra of the following Sturm-Liouville problems:

\begin{equation}\label{SL2}
\begin{cases}
-u_{k,l}''(r)-(n-1)\frac{\Theta'(r)}{\Theta(r)}u_{k,l}'(r)+\frac{l(l+n-2)}{\Theta^2(r)}u_{k,l}(r)=\lambda_{k,l} u_{k,l}(r)\,, & r\in(0,D)\\
u_{k,0}'(0)=u_{k,0}'(D)=0 &  {\rm or}\\
u_{k,l}(0)=u_{k,l}(D)=0 & {\rm if\ }l\ne 0.
\end{cases}
\end{equation}
The eigenfunctions are expressed by separation of variables as in the proof of Proposition \ref{rev_bdry}. In particular,  $\lambda_{1,0}=0$ is the first eigenvalue of $M$; a corresponding eigenfunction is $u_{1,0}(r)\equiv 1$. The second eigenvalue $\lambda$  is positive. Let $u$ be any function in the eigenspace associated with $\lambda$.

\medskip

Now one of the following things may happen:

\begin{enumerate}[1)]
\item $\lambda=\lambda_{2,0}$, $u(r,t)=u_{2,0}(r)$; the eigenvalue is simple and radial.
\item  $\lambda=\lambda_{1,1}$, $u(r,t)=u_{1,1}(r)H_1(t)$ for some spherical harmonic $H_1(t)$ of degree $1$;
the eigenvalue has multiplicity $n$.
\item $\lambda=\lambda_{2,0}=\lambda_{1,1}$, $u=a u_{2,0}(r)+b u_{1,1}(r)H_1(t)$ for some spherical harmonic $H_1(t)$ of degree $1$ and $a,b\in\mathbb R$; the eigenvalue has multiplicity $n+1$.
\end{enumerate}

At any rate, the second eigenfunction $u$ has two nodal domains (this is a consequence of standard Sturm-Liouville theory). We prove now the theorem by inspecting each different case.
\begin{enumerate}[1)]
\item If  $\lambda=\lambda_{2,0}$, then $u=u_{2,0}(r)$. In particular, $\lambda$ is the first Dirichlet eigenvalue for each of the two nodal domains $M_{\pm}$, which are manifolds of revolution with poles $\pm x_0$. Consequently, $u$ is obtained by joining in a proper way the two first Dirichlet eigenfunctions on $M_{\pm}$. In particular, $u$ has two critical points from Proposition \ref{rev_bdry}. We remark that, even if $u_{2,0}(r)$ is not a second eigenfunction of $M$, it is some eigenfunction which has always exactly two critical points, a maximum and a minimum.
\item If $\lambda=\lambda_{1,1}$, then $u(r,t)=u_{1,1}(r)H_1(t)$. A sufficient condition to ensure that $u$  has two critical points is that the first eigenfunction $u_{1,1}$ of the following Sturm-Liouville problem
\begin{equation}\label{SL3}
\begin{cases}
-u_{1,1}''-(n-1)\frac{\Theta'}{\Theta}u_{1,1}'+\frac{(n-1)u_{1,1}}{\Theta^2}=\lambda u_{1,1}\\
u_{1,1}(0)=u_{1,1}(D)=0
\end{cases}
\end{equation}
has only one critical point, which is a maximum (or a minimum).  We have that $u_{1,1}$ does not change sign and is strictly positive in $(0,D)$. By assumption we have $\Theta''<0$ on $(0,D)$. We define $N(r):=\Theta^{n-1}(r)u_{1,1}'(r)$. Since $\Theta''<0$, $\Theta>0$, and $\Theta(0)=\Theta(D)=0$, we have that there exists a unique $R\in(0,D)$ such that $\Theta'(R)=0$, while $\Theta'(r)>0$ in $(0,R)$ and $\Theta'(r)<0$ in $(R,D)$. We have, using the differential equation in \eqref{SL3}, that
$$
N'(r)=\left(\frac{n-1}{\Theta^2(r)}-\lambda\right)\Theta^{n-1}(r)u_{1,1}(r).
$$
Clearly $N'(r)>0$ for $r\in(0,\delta)$, for some $\delta>0$. Moreover $\Theta^{-2}(r)$ is decreasing for $r\in(0,R)$ and increasing in $(R,D)$. Hence $N'$ has at most two zeros in $(0,D)$. Recall that $N(0)=N(D)=0$. If $N'$ has no zeros, it is always positive, which means that $N>0$, and therefore $u_{1,1}'>0$, which is not possible, since $u_{1,1}$ vanishes at $r=0$ and $r=D$. For the same reason $N'$ cannot have only one zero, otherwise $N$ would still be positive in $(0,D)$. Then $N'$ has two zeros. This implies that that $u_{1,1}'$ vanishes only once in $(0,D)$, and therefore $u_{1,1}$ has a unique maximum. Hence the second eigenfunction $u$ on $M$ has exactly two critical points, a maximum and a minimum since it is given by $u_{1,1}(r)$ multiplied by $H_1(t)$.
\item A basis of the eigenspace corresponding to $\lambda$ is given by \\$\{u_{2,0}(r),u_{1,1}(r)H_1^1(t),...,u_{1,1}(r)H_1^n(t)\}$, where $\{H_1^1(t)$,...,$H_1^n(t)\}$ is any basis of the space of spherical harmonics of degree $1$ in $\mathbb S^{n-1}$.
\end{enumerate}
\end{proof}

The previous result, and the fact that in dimension $n=2$ the quantity $-\frac{\Theta''}{\Theta}$ is the Gaussian curvature of a manifold of revolution, motivates the following

\begin{conj}
Let $M$ be a closed Riemannian surface diffeomorphic to $\mathbb S^2$ of positive Gaussian curvature, and let $u$ be a second eigenfunction of the Laplacian on $M$. Is it true that $u$ has two non-degenerate critical points, a maximum and a minimum?
\end{conj}

\bibliography{paperGPfin.bib}

\begin{thebibliography}{10}

\bibitem{App}
A.~Acker, L.~E. Payne, and G.~Philippin.
\newblock On the convexity of level lines of the fundamental mode in the
  clamped membrane problem, and the existence of convex solutions in a related
  free boundary problem.
\newblock {\em Z. Angew. Math. Phys.}, 32(6):683--694, 1981.

\bibitem{bcnsww}
T.~Bourni, J.~Clutterbuck, X.~H. Nguyen, A.~Stancu, G.~Wei, and V.-M. Wheeler.
\newblock The vanishing of the fundamental gap of convex domains in {$\Bbb
  {H}^n$}.
\newblock {\em Ann. Henri Poincar\'{e}}, 23(2):595--614, 2022.

\bibitem{bl}
H.~J. Brascamp and E.~H. Lieb.
\newblock On extensions of the {B}runn-{M}inkowski and {P}r\'{e}kopa-{L}eindler
  theorems, including inequalities for log concave functions, and with an
  application to the diffusion equation.
\newblock {\em J. Functional Analysis}, 22(4):366--389, 1976.

\bibitem{bss}
J.-P. Brasselet, J.~Seade, and T.~Suwa.
\newblock {\em Vector fields on singular varieties}, volume 1987 of {\em
  Lecture Notes in Mathematics}.
\newblock Springer-Verlag, Berlin, 2009.

\bibitem{bls}
L.~Buhovsky, A.~Logunov, and M.~Sodin.
\newblock Eigenfunctions with infinitely many isolated critical points.
\newblock {\em Int. Math. Res. Not. IMRN}, (24):10100--10113, 2020.

\bibitem{cc}
X.~Cabr\'{e} and S.~Chanillo.
\newblock Stable solutions of semilinear elliptic problems in convex domains.
\newblock {\em Selecta Math. (N.S.)}, 4(1):1--10, 1998.

\bibitem{cf2}
L.~A. Caffarelli and A.~Friedman.
\newblock Convexity of solutions of semilinear elliptic equations.
\newblock {\em Duke Math. J.}, 52(2):431--456, 1985.

\bibitem{cf3}
L.~A. Caffarelli and A.~Friedman.
\newblock Partial regularity of the zero-set of solutions of linear and
  superlinear elliptic equations.
\newblock {\em J. Differential Equations}, 60(3):420--433, 1985.

\bibitem{dg}
F.~De~Regibus and M.~Grossi.
\newblock On the number of critical points of the second eigenfunction of the
  {L}aplacian in convex planar domains.
\newblock {\em J. Funct. Anal.}, 283(1):Paper No. 109496, 22, 2022.

\bibitem{dgm}
F.~De~Regibus, M.~Grossi, and D.~Mukherjee.
\newblock Uniqueness of the critical point for semi-stable solutions in {$\Bbb
  R^2$}.
\newblock {\em Calc. Var. Partial Differential Equations}, 60(1):Paper No. 25,
  13, 2021.

\bibitem{df}
H.~Donnelly and C.~Fefferman.
\newblock Nodal sets for eigenfunctions of the {L}aplacian on surfaces.
\newblock {\em J. Amer. Math. Soc.}, 3(2):333--353, 1990.

\bibitem{eps}
A.~Enciso and D.~Peralta-Salas.
\newblock Eigenfunctions with prescribed nodal sets.
\newblock {\em J. Differential Geom.}, 101(2):197--211, 2015.

\bibitem{epss}
A.~Enciso, D.~Peralta-Salas, and S.~Steinerberger.
\newblock Prescribing the nodal set of the first eigenfunction in each
  conformal class.
\newblock {\em Int. Math. Res. Not. IMRN}, (11):3322--3349, 2017.

\bibitem{gnn}
B.~Gidas, W.~M. Ni, and L.~Nirenberg.
\newblock Symmetry and related properties via the maximum principle.
\newblock {\em Comm. Math. Phys.}, 68(3):209--243, 1979.

\bibitem{gg4}
F.~Gladiali and M.~Grossi.
\newblock On the number of critical points of solutions of semilinear equations
  in {$\Bbb R^2$}.
\newblock {\em Amer. J. Math.}, 144(5):1221--1240, 2022.

\bibitem{hs}
R.~Hardt and L.~Simon.
\newblock Nodal sets for solutions of elliptic equations.
\newblock {\em J. Differential Geom.}, 30(2):505--522, 1989.

\bibitem{khan}
G.~Khan, X.~H. Nguyen, M.~Turkoen, and G.~Wei.
\newblock {L}og-{C}oncavity and {F}undamental {G}aps on {S}urfaces of
  {P}ositive {C}urvature.
\newblock {\em ar{X}iv:2211.06403}, 2022.

\bibitem{KoLe}
N.~J. Korevaar and J.~L. Lewis.
\newblock Convex solutions of certain elliptic equations have constant rank
  {H}essians.
\newblock {\em Arch. Rational Mech. Anal.}, 97(1):19--32, 1987.

\bibitem{lw}
Y.~I. Lee and A.~N. Wang.
\newblock Estimate of {$\lambda_2-\lambda_1$} on spheres.
\newblock {\em Chinese J. Math.}, 15(2):95--97, 1987.

\bibitem{ml}
L.~G. Makar-Limanov.
\newblock The solution of the {D}irichlet problem for the equation {$\Delta
  u=-1$}\ in a convex region.
\newblock {\em Mat. Zametki}, 9:89--92, 1971.

\bibitem{ms}
M.~Mukherjee and S.~Saha.
\newblock Nodal sets of {L}aplace eigenfunctions under small perturbations.
\newblock {\em Math. Ann.}, 383(1-2):475--491, 2022.

\bibitem{myers}
S.~B. Myers.
\newblock Isometries of 2-dimensional riemannian manifolds into themselves.
\newblock {\em Proceedings of the National Academy of Sciences of the United
  States of America}, 22(5):297--300, 1936.

\bibitem{nty}
N.~Nadirashvili, D.~Tot, and D.~Jakobson.
\newblock Geometric properties of eigenfunctions.
\newblock {\em Uspekhi Mat. Nauk}, 56(6(342)):67--88, 2001.

\bibitem{petersen}
P.~Petersen.
\newblock {\em Riemannian geometry}, volume 171 of {\em Graduate Texts in
  Mathematics}.
\newblock Springer, Cham, third edition, 2016.

\bibitem{s}
Y.~Shih.
\newblock A counterexample to the convexity property of the first eigenfunction
  on a convex domain of negative curvature.
\newblock {\em Comm. Partial Differential Equations}, 14(7):867--876, 1989.

\bibitem{swyy}
I.~M. Singer, B.~Wong, S.-T. Yau, and S.~S.-T. Yau.
\newblock An estimate of the gap of the first two eigenvalues in the
  {S}chr\"{o}dinger operator.
\newblock {\em Ann. Scuola Norm. Sup. Pisa Cl. Sci. (4)}, 12(2):319--333, 1985.

\bibitem{sp}
R.~P. Sperb.
\newblock {\em Maximum principles and their applications}, volume 157 of {\em
  Mathematics in Science and Engineering}.
\newblock Academic Press, Inc. [Harcourt Brace Jovanovich, Publishers], New
  York-London, 1981.

\bibitem{w}
F.-Y. Wang.
\newblock On estimation of the {D}irichlet spectral gap.
\newblock {\em Arch. Math. (Basel)}, 75(6):450--455, 2000.

\bibitem{weinb}
H.~F. Weinberger.
\newblock Remark on the preceding paper of {S}errin.
\newblock {\em Arch. Rational Mech. Anal.}, 43:319--320, 1971.

\bibitem{yau2}
S.~T. Yau.
\newblock A note on the distribution of critical points of eigenfunctions.
\newblock In {\em Tsing {H}ua lectures on geometry \& analysis ({H}sinchu,
  1990--1991)}, pages 315--317. Int. Press, Cambridge, MA, 1997.

\end{thebibliography}
\bibliographystyle{abbrv}

\end{document}